\documentclass[letterpaper, 10 pt, conference]{ieeeconf}  

\IEEEoverridecommandlockouts                              

\usepackage{amsmath,amssymb}
\usepackage{mathtools}
\usepackage{amsthm}
\usepackage{amsfonts}
\usepackage{cite}
\usepackage{comment}
\usepackage{multirow}

\newcommand{\randVar}{\boldsymbol{w}}
\newcommand{\desVar}{\boldsymbol{u}}
\newcommand{\statVar}{\boldsymbol{x}}
\newcommand{\costHead}{c}
\newcommand{\distanceHead}{d}
\newcommand{\cost}[2]{\costHead(#1,\,#2)}
\newcommand{\distance}{\distanceHead(\refProbHead,\,\realProbHead)}
\newcommand{\distanceT}[1]{\distanceHead(\refProbHead,\,\realProbHead_{#1})}

\newcommand{\dParameter}{\gamma}

\newcommand{\setRand}{\mathcal{W}}
\newcommand{\setDes}{\mathcal{U}}
\newcommand{\setStat}{\mathcal{X}}

\newcommand{\dimRand}{n}
\newcommand{\dimDes}{m}
\newcommand{\dimStat}{n}

\newcommand{\real}{\mathbb{R}}
\newcommand{\realPlus}{[0,\,\infty)}
\newcommand{\realPlusPlus}{(0,\,\infty)}

\newcommand{\refProbHead}{{P}_0}
\newcommand{\refProb}{\refProbHead(\randVar)}
\newcommand{\realProbHead}{P}
\newcommand{\realProb}{\realProbHead(\randVar)}
\newcommand{\realProbTHead}[1]{\realProbHead_{#1}}
\newcommand{\realProbT}[1]{\realProbTHead{#1}(\randVar_{#1})}

\newcommand{\probabilityHead}{\mathcal{P}}
\newcommand{\probability}[1]{\probabilityHead\left(#1\right)}

\newcommand{\probValue}[2]{\mathbb{P}_{#1}\left[#2\right]}
\newcommand{\expectation}[2]{\mathbb{E}_{#1}\left[#2\right]}
\newcommand{\variance}[2]{\mathbb{V}_{#1}\left[#2\right]}

\newcommand{\anySet}{S}
\newcommand{\anyRandom}{s}
\newcommand{\anyCondition}{A}
\newcommand{\anyFunction}{f}

\newcommand{\driftHead}{f}
\newcommand{\drift}{f(\statVar,\,\desVar,\,\randVar)}
\newcommand{\valueDRCHead}{V_\mathrm{DRC}}
\newcommand{\valueMVHead}{V}
\newcommand{\valueDRC}{\valueDRCHead(\statVar)}
\newcommand{\valueMV}{\valueMVHead(\statVar)}

\theoremstyle{definition}
\newtheorem{dfn}{Definition}

\newtheorem{prop}[dfn]{Proposition}

\newtheorem{thm}[dfn]{Theorem}
\newtheorem{cor}[dfn]{Corollary}
\newtheorem{rem}[dfn]{Remark}

\title{\LARGE \bf
Avoiding Semi-Infinite Programming in Distributionally Robust Control Based on Mean--Variance Metrics*
}

\author{Yuma Shida$^{1,**}$ and Yuji Ito$^{1}$
\thanks{*This work has been submitted to the IEEE for possible publication. Copyright may be transferred without notice, after which this version may no longer be accessible. This work was not supported by any organization.}
\thanks{$^{1}$Toyota Central R\&D Labs., Inc., 41-1, Yokomichi, Nagakute, Aichi, 480-1192, Japan.         {\tt\small \{Yuma.shida.fw, ito-yuji\}@mosk.tytlabs.co.jp} }%
\thanks{$^{**}$Permanent affiliation: Toyota Motor Corporation, 1, Toyota-Cho, Toyota City, Aichi, 471-8571, Japan. {\tt\small yuma\_shida@mail.toyota.co.jp} }%
}

\begin{document}

\maketitle
\thispagestyle{empty}
\pagestyle{empty}

\begin{abstract}
Conventional stochastic control methods have several limitations. They focus on optimizing the average performance and, in some cases, performance variability; however, their problem settings still require an explicit specification of the probability distributions that determine the system's stochastic behavior. 
Distributionally robust control (DRC) methods have recently been developed to address these challenges. However, many DRC approaches involve handling infinitely many inequalities. For instance, DRC problems based on the Wasserstein distance are commonly obtained by solving semi-infinite programming (SIP) problems. 
Our proposed method eliminates the need for SIP when solving discrete-time, discounted, distributionally robust optimal control problems. By introducing a penalty term based on a specific distributional distance, we establish upper bounds, and under appropriate conditions, demonstrate the equivalence between distributionally robust optimization problems and mean--variance minimization problems. This reformulation reduces the original DRC problem to a discounted mean--variance cost optimization problem. In linear-quadratic regulator settings, the corresponding control laws are obtained by solving the Riccati equation. Numerical experiments demonstrate that the theoretical maximum value of the discounted cumulative cost for the proposed method is lower than that for the conventional method.
\end{abstract}

\section{INTRODUCTION}
Real-world systems often encounte uncertainties such as randomness. Stochastic optimal control (SOC) is an approach for optimizing control policies in systems subjected to random disturbances \cite{bertsekas1996stochastic,crespo2003stochastic}. For example, methods based on Bellman's principle of optimality make it more tractable to compute control solutions for systems with state and cost-function constraints. 

Many SOC methods have primarily focused on optimizing the expected performance, such as the expected cost under a known probability distribution. However, optimizing expectations often results in unsatisfactory outcomes. To address this limitation, risk-sensitive control methods \cite{nishimura2021rat,kishida2022risk,whittle1981risk,petersen2002minimax} has been developed, which accounts for the higher-order characteristics of the distribution, such as variance, in addition to expectations.

In contrast, when the true probability distribution of the system is unknown, distributionally robust control (DRC) \cite{kishida2022risk,taskesen2023distributionally,liu2023data,yang2020wasserstein} has emerged as an alternative approach. A limitation of risk-sensitive and other distribution-dependent SOC methods \cite{bertsekas1996stochastic,crespo2003stochastic,whittle1981risk} is that these problem settings require the true system distribution to be known to compute expectations and variances. To address distributional uncertainty, previous study have also explored formulations based on the worst-case conditional value at risk (CVaR) \cite{kishida2022risk}.

However, DRC still faces significant computational challenges. DRC methods based on general distributional distance metrics, such as the optimal transport distance, typically lead to semi-infinite programming (SIP) \cite{reemtsen1991discretization,reemtsen2013semi}, which are difficult to solve \cite{yang2020wasserstein}. This issue also arises in distributionally robust optimization (DRO) problems \cite{reemtsen2013semi,shafieezadeh2019regularization,cherukuri2022data,mehrotra2014cutting,luo2017decomposition,mohajerin2018data,wiesemann2014distributionally}. Avoiding such SIP while guaranteeing robustness against distributional uncertainty is both theoretically and practically important. The worst-case CVaR-based linear-quadratic DRC method \cite{kishida2022risk} avoids SIP by using semidefinite programming; however, this requires fixing the first and second moments of all distributions.

The proposed method introduces a DRC formulation based on a penalty term associated with a specific distributional distance, which eliminates the need for SIP. The proposed method leverages mean--variance-type Bellman equations with respect to a reference distribution, thereby enabling the computation of distributionally robust controllers even when the distribution is unknown.

The contributions of this study are as follows:
\begin{itemize}
\item Formulation without SIP: We first address solutions of the DRO problems. We demonstrate that incorporating the variance of the cost is equivalent to optimizing the expected cost or its upper bound under the worst-case distribution with a distributional distance penalty. Furthermore, we show that solutions to the DRC for discounted cumulative costs in an infinite-horizon discrete-time setting with distributional uncertainty can be characterized by mean--variance-type Bellman equations. This formulation completely avoids SIP.  
\item Solution via the Riccati equation: 
For linear-quadratic control problems, the proposed method enables controller synthesis via solutions to DRC-type Riccati algebraic equations. Unlike prior studies \cite{yang2020wasserstein,petersen2002minimax,nishimura2021rat}, which established this property for continuous distributions, our approach extends the theory to discrete distributions, ensuring robustness under distributional uncertainty. This extension is validated throughout numerical experiments 
using an inverted pendulum on a cart \cite{kishida2022risk}.
\end{itemize}

This study partially builds on concepts from previous studies \cite{11107818,shida2025explicitreformulationdiscretedistributionally} while avoiding SIP. By using a continuously differentiable function as the distributional distance penalty, we reformulate the original min-max problem into a single-stage minimization, enabling efficient computation of the solution. Unlike previous studies, the method proposed in this paper relaxes the traditional ball constraint into a distance penalty and handles system dynamics constraints over a continuous state space and an infinite horizon. Unlike conventional DRC problems based on DRC-type Bellman equations \cite{liu2023data,neufeld2023markov,yu2023fast,liu2022distributionally}, the proposed method reduces DRC problems to mean--variance-type Bellman equations, which are single-layer minimization problems and structurally differ from traditional DRC formulations. 

The remainder of this paper is organized as follows.
Section \ref{sec:setting} introduces the problem settings. Section \ref{sec:proposed method} presents our proposed method. 
Section \ref{sec:non-global} shows its equivalence to a corresponding mean--variance-type minimization problem in the context of DRO. 
Section \ref{sec:global} addresses the nonlinear DRC settings, while Section \ref{sec:globalLQ} addresses the linear-quadratic cases \cite{kishida2022risk,whittle1981risk,taskesen2023distributionally,nishimura2021rat,mohajerin2018data,hajar2024distributionally}. 
The corresponding proofs are presented in Section \ref{sec:thm}. Section \ref{sec:experiment} verifies the theoretical results of the proposed method through numerical experiments, and Section \ref{sec:conc} concludes the paper. 

\section*{Notation}\label{sec:settingNotation}
The symbols $\real$, $\real^n$, and $\real^{n\times m}$ denote the sets of real numbers, $n$-dimensional real vectors, and real matrices of size $n\times m$, respectively. The set of all probability distributions on a set $\anySet$ is denoted by $\probability{\anySet}$. The probability of any condition $\anyCondition(\anyRandom)$ under distribution $\realProbHead(\anyRandom)$ is written as $\probValue{\realProbHead(\anyRandom)}{\anyCondition(\anyRandom)}$, where $\anyRandom$ is an arbitrary random variable. The expectation and variance of any function $\anyFunction(\anyRandom)$ are denoted as $\expectation{\realProbHead(\anyRandom)}{\anyFunction(\anyRandom)}$ and $\variance{\realProbHead(\anyRandom)}{\anyFunction(\anyRandom)}$, respectively. For any $\anyRandom\in\real$, $\anyRandom^+ = \max\{\anyRandom,\,0\}$. Let $I$ denotes the identity matrix. 

\section{Problem Settings}\label{sec:setting}
Section \ref{sec:settingDRO} presents  the DRO problem setting involving a single cost. Section \ref{sec:settingDRC} presents the DRC problem setting, which considers the discounted cumulative cost under the system dynamics constraint over an infinite horizon.

\subsection{DRO Problem}\label{sec:settingDRO}
We consider the DRO problem with a distributional distance penalty as follows:
\begin{equation}\label{eq:DRO}
\min_{\desVar\in\setDes} \quad \max_{\realProbHead\in\probability{\setRand}} \quad  \expectation{\realProb}{\cost{\randVar}{\desVar}}-\dParameter\,\distance.
\end{equation}
Here, $\costHead:\real^\dimRand\times\real^\dimDes\rightarrow\realPlus$ and $\distanceHead:\probability{\setRand}\times\probability{\setRand}\rightarrow\realPlus$ denote a cost function and a distributional distance function, respectively. The random variable $\randVar$ belongs to $\setRand\subseteq\real^\dimRand$, and the decision variable $\desVar$ belongs to $\setDes\subseteq\real^\dimDes$. The distributions $\refProbHead,\,\realProbHead\in\probability{\setRand}$ denote a reference distribution and a candidate worst-case distribution, respectively. The parameter $\dParameter\in\realPlusPlus$ is a positive coefficient associated with the distributional distance penalty.

\subsection{DRC Problems}\label{sec:settingDRC}
We define the DRC problem as the minimization of  the discounted cumulative cost with a distributional distance penalty over $(\desVar_0,\,\desVar_1,\,\cdots)\in\setDes^\infty$, where the objective function is defined as follows:
\begin{equation}\label{eq:DRC}
\begin{split}
&\cost{\statVar_0}{\desVar_0}
+\max_{(\realProbTHead{1},\,\realProbTHead{2},\,\cdots)\in\probability{\setRand}^\infty} \\
&\expectation{\prod_{t=1}^{\infty}\realProbT{t}}{\sum_{t=1}^{\infty}\alpha^t
\cost{\statVar_t}{\desVar_t}} 
-\dParameter\sum_{t=1}^{\infty}\alpha^{t-1}\distanceT{t}, \\
&\quad\quad\quad
\text{s.t.}\quad \statVar_{t+1}=f(\statVar_{t},\,\desVar_{t},\,\randVar_{t+1}). \raisetag{11pt}
\end{split}
\end{equation}
Here, $\statVar_t\in\setStat\subseteq\real^\dimStat$, $\randVar_t\in\setRand\subseteq\real^\dimRand$, and $\desVar_t\in\setDes\subseteq\real^\dimDes$ represent the state variable, random variable, and decision variable at each time $t\in\{0,\,1,\,\cdots\}$, respectively. The initial state is denoted by $\statVar_0\in\setStat$. The distribution $\realProbT{t}$ is a candidate worst-case distribution of $\randVar_t$ at time $t$. 
While $\randVar_t$ is independent of $\randVar_s$ ($s\neq t$), the distributions $\realProbT{t}$ and $\realProbT{s}$ are not necessarily identical. The parameter $\alpha\in(0,\,1)$ is a discount factor. The function $\driftHead:\setStat\times\setDes\times\setRand\rightarrow\setStat$  represents the system dynamics.

Furthermore, 
(\ref{eq:DRC}) can be reformulated into the corresponding Bellman equation as follows:
\begin{equation}\label{eq:DRCBellman}
\begin{split}
\valueDRC &= \min_{\desVar\in\setDes}\quad\cost{\statVar}{\desVar} 
+\max_{\realProbHead\in\probability{\setRand}}\quad \\
&\alpha\,\expectation{\realProb}{\valueDRCHead(\drift)} 
 -\dParameter\distance. \\
\end{split}
\end{equation}
This reformulation can be stated as the following proposition by using the value function $\valueDRCHead:\setStat\rightarrow\realPlus$.

\begin{prop}[Bellman Equations of DRC Problems]\label{prop:DRC}
Suppose that there exists a value function $\valueDRC$ that satisfies the Bellman equation (\ref{eq:DRCBellman}), that $\desVar^*(\statVar)$ is a minimizer to (\ref{eq:DRCBellman}) for every $\statVar\in\setStat$, and that the discount factor $\alpha$ is sufficiently small such that $\lim_{T\rightarrow\infty}\alpha^T\expectation{\prod_{t=1}^{T}\realProbT{t}}{\valueDRCHead(\statVar_T)}=0$ for $\desVar_{t}=\desVar^*(\statVar_t)$ and any $\realProbT{t+1}\in\probability{\setRand}$, for all $t\in\{0,\,1,\,\cdots\}$, under $\statVar_{t+1}=f(\statVar_{t},\,\desVar_{t},\,\randVar_{t+1})$. Then, the DRC problem (\ref{eq:DRC}) satisfies the following properties:
\begin{enumerate}
\renewcommand{\labelenumi}{(\roman{enumi})}
\item Value function $\valueDRCHead(\statVar_0)$ denotes the unique optimal worst-case value of the discounted cumulative cost (\ref{eq:DRC}).
\item The minimizer to (\ref{eq:DRC}) is $(\desVar^*(\statVar_0),\,\desVar^*(\statVar_1),\,\cdots)$.
\end{enumerate}  
\end{prop}

\begin{proof}[Proof of Proposition \ref{prop:DRC}]
This proof relies on different assumptions but can be established based on \cite[Theorem 2.7 (iii)]{neufeld2023markov}. In particular, the value function $\valueDRCHead(\statVar_0)$ that satisfies (\ref{eq:DRCBellman}) can be inductively rewritten by substituting $\statVar=\statVar_0$, $\desVar=\desVar_0$, $\desVar_t=\desVar^*(\statVar_t)$, $\realProbHead(\randVar)=\realProbHead_1(\randVar_1)$, and $f(\statVar_{t},\,\desVar_t,\,\randVar_{t+1})=\statVar_{t+1}$ for every $t\in\{1,\,2,\,\cdots\}$, as follows:
\begin{equation*}
\begin{split}
&\valueDRCHead(\statVar_0)=
\cost{\statVar_0}{\desVar^*(\statVar_0)}+\max_{(\realProbTHead{1},\,\realProbTHead{2},\,\cdots)\in\probability{\setRand}^\infty} \\
&\lim_{T\rightarrow\infty}\expectation{\prod_{t=1}^{T}\realProbT{t}}{\sum_{t=1}^{T-1}\alpha^t
\cost{\statVar_t}{\desVar^*(\statVar_t)}+\alpha^T\valueDRCHead(\statVar_T)} \\
&-\dParameter\sum_{t=1}^{T}\alpha^{t-1}\distanceT{t}.
\end{split}
\end{equation*}
From the assumptions introduced in Proposition \ref{prop:DRC}, we can replace $\lim_{T\rightarrow\infty}\alpha^T\expectation{\prod_{t=1}^{T}\realProbT{t}}{\valueDRCHead(\statVar_T)}$ with zero, and $(\desVar^*(\statVar_0),\,\desVar^*(\statVar_1),\,\cdots)$ becomes a minimizer over $(\desVar_0,\,\desVar_1,\,\cdots)\in\setDes^\infty$ in the last equation. 
The statements (i) and (ii) can be proven individually.
\end{proof}


\section{Proposed Method}\label{sec:proposed method}
Sections \ref{sec:non-global} and \ref{sec:global} reformulate the DRO and DRC problems as single-layer mean--variance minimization and mean--variance-type control problems, respectively. These reformulations improve tractability. Section \ref{sec:globalLQ} presents a special case of DRC reformulation that can be solved entirely using algebraic equations.

\subsection{Computing Solutions to DRO problem}\label{sec:non-global}
This section aims to compute solutions to the DRO problem (\ref{eq:DRO}) in Section \ref{sec:settingDRO}. To make the problem more tractable, we consider the following mean--variance minimization problem with respect to the reference distribution:
\begin{equation}\label{eq:MV}
\min_{\desVar\in\setDes} \quad \expectation{\refProb}{\cost{\randVar}{\desVar}}+\frac{1}{4\dParameter}\variance{\refProb}{\cost{\randVar}{\desVar}}.
\end{equation}
 
We now introduce the following theorem: The DRO problem (\ref{eq:DRO}) can be reformulated into the equivalent mean--variance minimization problem (\ref{eq:MV}). 
\begin{thm}[Reformulation of DRO Problem]\label{thm:DRO}
Suppose that there exists a finite set $\setRand_N=\{\randVar^1,\,\cdots,\,\randVar^N\}$ such that $\setRand=\setRand_N$, and that the distance $\distanceHead$ in (\ref{eq:DRO}) is defined as: 
\begin{equation}\label{eq:penalty}
\distance\coloneq\expectation{\refProb}{\left(1-\frac{\realProb}{\refProb}\right)^2}.
\end{equation}
Then, the optimization problems (\ref{eq:DRO}) and (\ref{eq:MV}) satisfy the following properties:
\begin{enumerate}
\renewcommand{\labelenumi}{(\roman{enumi})}
\item For any decision variable $\desVar\in\setDes$, the objective in (\ref{eq:MV}) provides an upper bound for the objective in (\ref{eq:DRO}) as follows:
\begin{multline}\label{eq:upper}
\expectation{\refProb}{\cost{\randVar}{\desVar}}+\frac{1}{4\dParameter}\variance{\refProb}{\cost{\randVar}{\desVar}} \\ \geq \max_{\realProbHead\in\probability{\setRand}} \quad  \expectation{\realProb}{\cost{\randVar}{\desVar}}-\dParameter\,\distance.
\end{multline}
\item For any decision variable $\desVar\in\setDes$, if 
\begin{equation}\label{eq:upper2}
\min_{\randVar\in\setRand}\cost{\randVar}{\desVar}-\expectation{\refProb}{\cost{\randVar}{\desVar}}+2\dParameter>0,
\end{equation}
then the equality stated in the equation (\ref{eq:upper}) holds.
\end{enumerate}
\end{thm}

\begin{rem}[Proof of Theorem \ref{thm:DRO}]
See Section \ref{sec:thm}.
\end{rem}

\begin{rem}[Equivalence Between DRO and Mean--Variance Minimization Problems]
Theorem \ref{thm:DRO} clarifies the relationship between the DRO and mean--variance minimization problems: the statement (i) shows that the mean--variance minimization problem provides an upper bound for the DRO problem, while the statement (ii) shows that 
the mean and variance of the cost coincide with the maximum cost under the distance penalty   
when $\dParameter$ is sufficiently large. 
\end{rem}

\subsection{Computing Solutions to DRC}\label{sec:global}
This section computes solutions to DRC problem in Section \ref{sec:settingDRC}. To make the problem more tractable, we consider the mean--variance-type control problem that minimizes the following objective over $(\desVar_0,\,\desVar_1,\,\cdots)\in\setDes^\infty$.
\begin{equation}\label{eq:MVC}
\begin{split}
&\cost{\statVar_0}{\desVar_0}
+\rho^{(1,\,\infty)}\left(\sum_{t=1}^{\infty}\alpha^t
\cost{\statVar_t}{\desVar_t}\right) \\
&\quad\quad\quad\quad\quad\quad
\text{s.t.}\quad \statVar_{t+1}=f(\statVar_{t},\,\desVar_{t},\,\randVar_{t+1}),\raisetag{12pt}
\end{split}
\end{equation}
where $\rho^{(i,\,j)}\coloneq(\rho_{i} \circ
\rho_{i+1} \circ \cdots \circ \rho_{j})
$ denotes the composition of the functions used to compute the mean and variance of the cost as follows:
\begin{multline}\label{eq:mean--variance}
\rho_{i}\left(\sum_{t=1}^{\infty}\alpha^t
\cost{\statVar_t}{\desVar_t}\right) 
\coloneq
\expectation{\refProbHead(\randVar_i)}{\sum_{t=1}^{\infty}\alpha^t
\cost{\statVar_t}{\desVar_t}} \\
+\frac{1}{4\dParameter\,\alpha^{i-1}}\variance{\refProbHead(\randVar_i)}{\sum_{t=1}^{\infty}\alpha^t
\cost{\statVar_t}{\desVar_t}}.
\end{multline}

\begin{rem}[Mean--Variance Function]
The variance term of $\rho_{i}$ has a $1/\alpha^{i-1}$ coefficient, but all variance terms of order lower than $i$ vanish; namely, $\variance{\refProbHead(\randVar_i)}{\sum_{t=1}^{\infty}\alpha^t\cost{\statVar_t}{\desVar_t}}=\variance{\refProbHead(\randVar_i)}{\sum_{t=i}^{\infty}\alpha^{t}\cost{\statVar_{t}}{\desVar_{t}}}$. Thus, the output of $\rho^{(1,\,\infty)}$ in  (\ref{eq:MVC}) does not diverge due to the $1/\alpha^{i-1}$ coefficient. 
\end{rem}

Furthermore, 
(\ref{eq:MVC}) can be reformulated as the corresponding Bellman equation as follows:
\begin{equation}\label{eq:MVBellman}
\begin{split}
\valueMV =& \min_{\desVar\in\setDes}\quad
\cost{\statVar}{\desVar} +\alpha\,\expectation{\refProb}{\valueMVHead(\drift)} \\
&+\frac{1}{4\dParameter}\,\variance{\refProb}{\alpha\,\valueMVHead(\drift)}.
\end{split}
\end{equation}
This reformulation can be stated as the following proposition by using the value function $\valueMVHead:\setStat\rightarrow\realPlus$. 

\begin{prop}[Bellman Equations of Mean--Variance-Type Control Problems]\label{prop:MVC}
Suppose that there exists a value function $\valueMV$ that satisfies the mean--variance-type Bellman equation (\ref{eq:MVBellman}), that $\desVar^*(\statVar)$ is a minimizer to (\ref{eq:MVBellman}) for every $\statVar\in\setStat$, and that the discount factor $\alpha$ is sufficiently small such that $\lim_{T\rightarrow\infty}\rho^{(1,\,T)}\left(\sum_{t=1}^{T}\alpha^t
\cost{\statVar_t}{\desVar_t}+\alpha^{T+1}\valueMVHead(\statVar_{T+1})\right)=\rho^{(1,\,\infty)}\left(\sum_{t=1}^{\infty}\alpha^t
\cost{\statVar_t}{\desVar_t}\right)$, for $\desVar_{t}=\desVar^*(\statVar_{t})$, for all $t\in\{0,\,1,\,\cdots\}$, under $\statVar_{t+1}=f(\statVar_{t},\,\desVar_{t},\,\randVar_{t+1})$. Subsequently, the mean--variance-type control problem (\ref{eq:MVC}) satisfies the following properties:
\begin{enumerate}
\renewcommand{\labelenumi}{(\roman{enumi})}
\item Value function $\valueMVHead(\statVar_0)$ denotes the unique optimal mean--variance of the discounted cumulative cost (\ref{eq:MVC}).
\item The minimizer to (\ref{eq:MVC}) is $(\desVar^*(\statVar_0),\,\desVar^*(\statVar_1),\,\cdots)$.
\end{enumerate}
\end{prop}

\begin{proof}[Proof of Proposition \ref{prop:MVC}]
This proof differs from \cite{neufeld2023markov}[Theorem 2.7 (iii)] in that it is associated with the mean--variance-type Bellman equation: The value function $\valueMVHead(\statVar_0)$ that satisfies (\ref{eq:MVBellman}) can inductively be rewritten by substituting $\statVar=\statVar_0$, $\desVar=\desVar_0$, $\desVar_t=\desVar^*(\statVar_0)$, $\randVar=\randVar_1$, and $f(\statVar_{t},\,\desVar_t,\,\randVar_{t+1})=\statVar_{t+1}$ for every $t\in\{1,\,2,\,\cdots\}$, as follows:
\begin{equation*}
\begin{split}
&\valueMVHead(\statVar_0)=
\cost{\statVar_0}{\desVar^*(\statVar_0)} \\
&+\lim_{T\rightarrow\infty}\rho^{(1,\,T)}\left(\sum_{t=1}^{T}\alpha^t
\cost{\statVar_t}{\desVar^*(\statVar_t)}+\alpha^{T+1}\valueMVHead(\statVar_{T+1})\right).\end{split}
\end{equation*}
From the assumption introduced in Proposition \ref{prop:MVC}, we can replace the second term in the right-hand side of the last equation with $\rho^{(1,\,\infty)}\left(\sum_{t=1}^{\infty}\alpha^t
\cost{\statVar_t}{\desVar_t}\right)$, and $(\desVar^*(\statVar_0),\,\desVar^*(\statVar_1),\,\cdots)$ becomes a minimizer over $(\desVar_0,\,\desVar_1,\,\cdots)\in\setDes^\infty$ in the last equation. Thus, the statements (i) and (ii)  can be proven individually. 
\end{proof}

From Theorem \ref{thm:DRO} (ii), we obtain the following corollary. 
\begin{cor}[Reformulations of DRC Problems]\label{cor:equivalence}
Suppose that the assumptions introduced in Theorem \ref{thm:DRO} and Propositions \ref{prop:DRC} and \ref{prop:MVC} hold, and that $\dParameter$ is sufficiently large such that the following inequality holds for any $(\statVar,\,\desVar)\in\setStat\times\setDes$. 
\begin{equation}\label{eq:assumption}
\begin{split}
&\forall\valueMVHead^\prime\in\{\valueDRCHead,\,\valueMVHead\}, \quad\min_{\randVar\in\setRand}\alpha\,\valueMVHead^\prime(\drift)\\
&\quad\quad\quad-\alpha\,\expectation{\refProb}{\valueMVHead^\prime(\drift)}+2\dParameter>0.
\end{split}
\end{equation}
Then, the following properties hold.
\begin{enumerate}
\renewcommand{\labelenumi}{(\roman{enumi})}
\item The Bellman equations (\ref{eq:DRCBellman}) and (\ref{eq:MVBellman}) are identical.
\item The DRC problem (\ref{eq:DRC}) and the  mean--variance-type control problem (\ref{eq:MVC}) are identical.
\end{enumerate}  
\end{cor}

\begin{proof}[Proof of Corollary \ref{cor:equivalence}]
By considering two cases in which $\cost{\randVar}{\desVar}$ in (\ref{eq:DRO}) and (\ref{eq:MV}) is replaced with $\cost{\statVar}{\desVar}+\alpha\,\valueDRCHead$ $(\drift)$ and $\cost{\statVar}{\desVar}+\alpha\,\valueMVHead(\drift)$, respectively, the statement (i) immediately follows from Theorem \ref{thm:DRO} (ii). 

Furthermore, from Propositions \ref{prop:DRC} (i) and \ref{prop:MVC} (i), 
$\valueDRC$ and $\valueMV$ are not only the unique optimal values of their respective problems but also identical; namely, $\valueDRC=\valueMV$. Therefore, according to the statement (i), the minimizers of the two Bellman equations are identical. From Propositions \ref{prop:DRC} (ii) and \ref{prop:MVC} (ii), these minimizers are identical to the minimizers to the DRC and mean--variance-type control problems, respectively. Hence, the two problems become identical, and the statement (ii) follows. 
\end{proof}

\begin{rem}[Equivalence Between DRC and Mean--Variance-Type Control Problems]\label{rem:equivalence}
Corollary \ref{cor:equivalence} reformulates the DRC problems as mean--variance-type control problems. The statement (i) rewrites the Bellman equation of the DRC problem as the Bellman equation of the mean--variance-type control problem. The rewritten Bellman equation is reduced to a single-stage minimization. Furthermore, the statement (ii) expresses the worst-case discounted cumulative cost in terms of the mean and variance of the discounted cumulative cost. This result indicates that the mean and variance can serve as theoretical maximum values for any controller.
\end{rem}

\subsection{Computing Solutions via Riccati Equations}\label{sec:globalLQ}
Under the case introduced in this section, solving the Bellman equation (\ref{eq:MVBellman}) is reduced to solving the corresponding Riccati equation. 
This approach enables solutions to be obtained more efficiently. We consider the quadratic cost function and linear system dynamics as follows: 
\begin{equation}\label{eq:costLQ}
\cost{\statVar}{\desVar}=\statVar^\top Q \statVar+\desVar^\top R \desVar,
\end{equation}
\begin{equation}\label{eq:dynamicsLQ}
\drift=A\statVar + B \desVar + \randVar. 
\end{equation}
Here, $Q\in\real^{\dimStat\times\dimStat}$ and $R\in\real^{\dimDes\times\dimDes}$ are the positive definite matrices defining the cost function.  $A\in\real^{\dimStat\times\dimStat}$ and $B\in\real^{\dimStat\times\dimDes}$ are the matrices that define the system dynamics. In addition, we denote the first and second moments of the reference probability distribution $\refProb$: the mean of $\randVar$ under $\refProb$ is zero, and its covariance is $\Sigma$. The information about the probability distribution $\realProb$ remains unknown.

Furthermore, we consider the Riccati equation corresponding to the Bellman equation (\ref{eq:MVBellman}). Under the case in (\ref{eq:costLQ}) and (\ref{eq:dynamicsLQ}), (\ref{eq:MVBellman}) is reduced to the following Riccati equation.
\begin{multline}\label{eq:riccati}
P^* = \alpha A^\top \tilde{P^*} A + Q \\ - \alpha^2 A^\top \tilde{P^*} B (R + \alpha B^\top \tilde{P^*} B)^{-1} B^\top \tilde{P^*} A. 
\end{multline}
Here, $P^*\in\real^{\dimStat\times\dimStat}$ is a symmetric positive semidefinite solution to (\ref{eq:riccati}). We define $\tilde{P^*}$ as $\tilde{P^*}\coloneq P^* + \frac{\alpha}{\dParameter} P^* \Sigma P^*.$

\begin{thm}[Solving Mean--Variance-Type Bellman Equation]\label{thm:MV}
Suppose that there exists a positive semidefinite symmetric matrix $P^*\in\real^{\dimStat\times\dimStat}$ that satisfies the Riccati equation (\ref{eq:riccati}). Then, the mean--variance-type Bellman equation (\ref{eq:MVBellman}) satisfies the following properties:
\begin{enumerate}
\renewcommand{\labelenumi}{(\roman{enumi})}
\item The following value function: 
\begin{equation}\label{eq:valueRiccati}
\valueMV=\statVar^\top P^* \statVar + r,
\end{equation}
satisfies (\ref{eq:MVBellman}), where $r=\frac{\alpha}{1-\alpha} \mathrm{Tr}[P^*\Sigma+\frac{\alpha}{2\dParameter} P^* \Sigma P^* \Sigma]$. 
\item The minimizer $\desVar^*(\statVar)$ to the right-hand side of (\ref{eq:MVBellman}) is given by:
\begin{equation}\label{eq:proposedController}
\desVar^*(\statVar)=-(R+\alpha B^\top \tilde{P^*} B)^{-1}\alpha B^{\top} \tilde{P^*} A \statVar.
\end{equation}
\end{enumerate}
\end{thm}

\begin{rem}[Proof of Theorem \ref{thm:MV}]
See Section \ref{sec:thm}.
\end{rem}

\begin{rem}[Solving Bellman Equations]
Theorem \ref{thm:MV} shows that the Bellman equation (\ref{eq:MVBellman}) is solvable. From Corollary \ref{cor:equivalence} (i), solving (\ref{eq:MVBellman}) is identical to solving the original Bellman equation (\ref{eq:DRCBellman}). 
\end{rem}

\begin{rem}[Solvability of Riccati Equations]\label{rem:solvability}
As the parameter $\dParameter$ approaches infinity, the Riccati equation (\ref{eq:riccati}) asymptotically matches that of the discounted linear-quadratic regulator \cite[Section 4.3]{BertsekasDPVol1_4th}. The solutions to the Riccati equations can be computed by finding the fixed point $P_\infty$, which satisfies the following difference equation.
\begin{multline}
P_{k+1} = \alpha A^\top \tilde{P_k} A + Q \\ - \alpha^2 A^\top \tilde{P_k} B (R + \alpha B^\top \tilde{P_k} B)^{-1} B^\top \tilde{P_k} A. 
\end{multline}
Here, $P_k\in\real^{\dimStat\times\dimStat}$ is a symmetric positive semidefinite matrix for each $k\in\{0,\,1,\,\cdots\}$, and $\tilde{P}_k$ is defined as $
\tilde{P}_k\coloneq P_k + \frac{\alpha}{\dParameter} P_k \Sigma P_k$. The matrix $P_k$ is obtained by applying the difference equation to the initial matrix $P_0$ $k$ times. 
\end{rem}

\begin{rem}[Statements Corresponding to Any Controllers]\label{rem:anyController}
If we replace $\desVar\in\setDes$ in Theorem \ref{thm:DRO}, Propositions \ref{prop:DRC} and \ref{prop:MVC}, and Corollary \ref{cor:equivalence} with $\desVar=\desVar^\prime(\statVar)$, the statements corresponding to any control law $\desVar^\prime(\statVar)$ can be obtained. In particular, the value function in (\ref{eq:valueRiccati}) is replaced with the following value function:
\begin{equation}\label{eq:valueLyapunov}
X(\statVar)=\statVar^\top Y\statVar+\frac{\alpha}{1-\alpha} \mathrm{Tr}[Y\Sigma+\frac{\alpha}{2\dParameter} Y \Sigma Y \Sigma].
\end{equation}
Here, $Y\in\real^{\dimStat\times\dimStat}$ is a symmetric positive semidifinite solution to the following Lyapunov equation with $\tilde{Y}\coloneq\frac{\alpha}{\dParameter}Y\Sigma Y$:
\begin{equation}\label{eq:lyapnov}
Y = \alpha (A-BK)^\top \tilde{Y} (A-BK) + Q + K^\top RK.
\end{equation}
The function $X(\statVar_0)$ serves as the theoretical maximum value of the discounted cumulative cost (\ref{eq:MVC}) for the controller $\desVar^\prime(\statVar)\coloneq-K\statVar$ with $K\in\real^{\dimDes\times\dimStat}$, since Corollary \ref{cor:equivalence} (ii) holds regardless of whether $\desVar$ is the minimizer or given by $\desVar=\desVar^\prime(\statVar)$.
\end{rem}

\subsection{Proofs of Theorems}\label{sec:thm}
\begin{proof}[Proof of Theorem \ref{thm:DRO}]
First, we prove the statement (i). From the assumptions introduced in Theorem \ref{thm:DRO}, the probability distribution $\refProb$ and $\realProb$ can be expressed as follows:
\begin{equation*}
\refProbHead(\randVar_i)=p_0^i,
\end{equation*}
\begin{equation*}
\realProbHead(\randVar_i)=p^i, 
\end{equation*}
where $p_0^i,\,p^i\in\realPlus$ are scalars satisfy $(p_0^1,\,\cdots,\,p_0^N),$ $\,(p^1,\,\cdots,\,p^N)\in\probabilityHead_N(\setRand_N)$. 
The set $\probabilityHead_N(\setRand_N)$ is defined as 
\begin{equation*}
\probabilityHead_N(\setRand_N) \coloneq \left\{(p^1,\,\cdots,\,p^N)\in[0,\infty)^N\mid\sum_{i=1}^N p^i = 1\right\}.
\end{equation*}
Then inner maximization problem of the  original DRO problem (\ref{eq:DRO}) can be rewritten as follows:
\begin{equation}\label{eq:rewrittenDRO}
\max_{(p^1,\,\cdots,\,p^N)\in\probabilityHead_N(\setRand_N)}  \sum_{i=1}^N \left( p^i\,\cost{\randVar^i}{\desVar}-\dParameter\,p_0^i\left(1-\frac{p^i}{p_0^i}\right)^2\right).
\end{equation}

This rewritten problem (\ref{eq:rewrittenDRO}) becomes a convex optimization problem that satisfies the Slater condition\cite[Section 4.2]{boyd2004convex}. Hence, by introducing the Lagrange dual problem\cite[5.2]{boyd2004convex} corresponding to (\ref{eq:rewrittenDRO}), we obtain the following problem:
\begin{equation}\label{eq:lagrangeDual}
\begin{split}
&\max_{(p^1,\,\cdots,\,p^N)\in\probabilityHead_N(\setRand_N)} 
 \sum_{i=1}^N \left( p^i\,\cost{\randVar^i}{\desVar}-\dParameter\,p_0^i\left(1-\frac{p^i}{p_0^i}\right)^2\right) \\
&= \inf_{s\in\real} \quad \max_{(p^1,\,\cdots,\,p^N)\in[0,\infty)^N} \quad \\
& \sum_{i=1}^N \left( p^i\,\cost{\randVar^i}{\desVar}-\dParameter\,p_0^i\left(1-\frac{p^i}{p_0^i}\right)^2\right) 
 + s\left(1-\sum_{i=1}^N p^i\right).
\end{split}
\end{equation}
Here, $s\in\real$ is a Lagrange multiplier corresponding to the constraint $1-\sum_{i=1}^N p^i$, which appears in the definition of $\probabilityHead_N(\setRand_N)$. 
The object function in the right-hand side of (\ref{eq:lagrangeDual}) is a quadratic function of $(p^1,\,\cdots,\,p^N)$. Thus, the maximizer ${p^i}^*$ for every $i\in\{1,\,\cdots,\,N\}$ is computed as ${p^i}^*={p_0^i} / {2\dParameter}\times(\cost{\randVar^i}{\desVar}+2\dParameter-s)^+$, and obtain 
\begin{equation}\label{eq:lagrangeDual2}
\begin{split}
&\max_{(p^1,\,\cdots,\,p^N)\in[0,\infty)^N} \quad \\
&\quad \sum_{i=1}^N \left( p^i\,\cost{\randVar^i}{\desVar}-\dParameter\,p_0^i\left(1-\frac{p^i}{p_0^i}\right)^2\right) 
 + s\left(1-\sum_{i=1}^N p^i\right) \\
&\quad=\dParameter\,\expectation{\refProb}{\left\{\left(\frac{\cost{\randVar^i}{\desVar}+2\dParameter-s}{2\dParameter} \right)^+\right\}^2 }-\gamma+s, \\
&\quad \leq \dParameter\,\expectation{\refProb}{\left(\frac{\cost{\randVar^i}{\desVar}+2\dParameter-s}{2\dParameter} \right)^2 }-\gamma+s. \raisetag{16pt}
\end{split}
\end{equation}

The final expression in (\ref{eq:lagrangeDual2}) is a  convex quadratic function of $s$. Thus, the minimizer $s^*$ can be computed as follows: 
\begin{equation}\label{eq:optimalMultiplier} s^*=\expectation{\refProb}{\cost{\randVar^i}{\desVar}}. 
\end{equation}
Hence, by substituting $s^*$ in place of $s$ in (\ref{eq:lagrangeDual}) and (\ref{eq:lagrangeDual2}), the following inequality holds:
\begin{equation*}
\begin{split}
&\inf_{s\in\real} \quad \max_{(p^1,\,\cdots,\,p^N)\in[0,\infty)^N} \quad \\
&\quad \sum_{i=1}^N \left( p^i\,\cost{\randVar^i}{\desVar}-\dParameter\,p_0^i\left(1-\frac{p^i}{p_0^i}\right)^2\right) 
 + s\left(1-\sum_{i=1}^N p^i\right) \\
&\quad \leq \expectation{\refProb}{\cost{\randVar^i}{\desVar}}+\frac{1}{4\dParameter}\variance{\refProb}{\cost{\randVar^i}{\desVar}}.
\end{split}
\end{equation*}
Thus, based on the (\ref{eq:rewrittenDRO}), 
the final expression matches the original problem (\ref{eq:rewrittenDRO}), thus, the statement (i) is proven.

Furthermore, we prove the statement (ii). By substituting $s^*$ in place of $s$ in (\ref{eq:lagrangeDual2}), we obtain the following equation. 
\begin{equation*}\label{eq:lagrangeDual4}
\begin{split}
&\max_{(p^1,\,\cdots,\,p^N)\in[0,\infty)^N} \quad \\
&\quad \sum_{i=1}^N \left( p^i\,\cost{\randVar^i}{\desVar}-\dParameter\,p_0^i\left(1-\frac{p^i}{p_0^i}\right)^2\right) 
 + s^*\left(1-\sum_{i=1}^N p^i\right) \\
& = \dParameter\,\expectation{\refProb}{\left\{ \left(\frac{\cost{\randVar^i}{\desVar}+2\dParameter-s^*}{2\dParameter}\right)^+ \right\}^2 }-\gamma+s^*.
\end{split}
\end{equation*}
From the assumption introduced in Theorem \ref{thm:DRO} (ii) and the definition of $s^*$ in (\ref{eq:optimalMultiplier}), the last equation can be reformulated as follows:
\begin{equation*}\label{eq:lagrangeDual5}
\begin{split}
&\dParameter\,\expectation{\refProb}{\left\{ \left(\frac{\cost{\randVar^i}{\desVar}+2\dParameter-s^*}{2\dParameter}\right)^+ \right\}^2 }-\gamma+s^* \\
&\quad= \dParameter\,\expectation{\refProb}{\left(\frac{\cost{\randVar^i}{\desVar}+2\dParameter-s^*}{2\dParameter} \right)^2 }-\gamma+s^* \\
&\quad=\expectation{\refProb}{\cost{\randVar^i}{\desVar}}+\frac{1}{4\dParameter}\variance{\refProb}{\cost{\randVar^i}{\desVar}}.
\end{split}
\end{equation*}
By observing that the right-hand side of (\ref{eq:lagrangeDual2}) is a convex function of $s$ and $s^*$ is the point that attains the extreme value of (\ref{eq:lagrangeDual2}), we conclude that the statement (ii) holds. 
\end{proof}

\begin{proof}[Proof of Theorem \ref{thm:MV}]
The theorem can be proved by substituting $\valueMV=\statVar^\top P^* \statVar + r$ into the Bellman equation (\ref{eq:MVBellman}).
First, based on (\ref{eq:dynamicsLQ}), we can write the following equation:
\begin{equation*}
\begin{split}
&\cost{\statVar}{\desVar}+\alpha\,\valueMVHead(\drift) = \alpha\,r + \\
&
\begin{bmatrix}
\statVar \\
\desVar \\
\randVar
\end{bmatrix}^\top
\begin{bmatrix}
\alpha\,A^\top P^* A + Q & \alpha\,A^\top P^* B & \alpha\,A^\top P^*  \\
\alpha\,B^\top P^* A & \alpha\,B^\top P^* B + R & \alpha\,B^\top P^* \\
\alpha\,P^* A & \alpha\,P^* B & \alpha\,P^* 
\end{bmatrix}
\begin{bmatrix}
\statVar \\
\desVar \\
\randVar
\end{bmatrix}
.
\end{split}
\end{equation*} 
Furthermore, by computing the expectation and variance of the last expression, the right-hand side of (\ref{eq:MVBellman}) is given by:
\begin{equation*}
\begin{split}
\min_{\desVar\in\setDes}\quad
&\begin{bmatrix}
\statVar \\
\desVar
\end{bmatrix}^\top
\begin{bmatrix}
\alpha\,A^\top \tilde{P^*} A + Q & \alpha\,A^\top \tilde{P^*} B \\
\alpha\,B^\top \tilde{P^*} A & \alpha\,B^\top \tilde{P^*} B + R \end{bmatrix}
\begin{bmatrix}
\statVar \\
\desVar
\end{bmatrix}
+r.
\end{split}
\end{equation*}
We can compute the minimizer $\desVar^*$ to the last equation as given in (\ref{eq:proposedController}). Finally, based on the minimizer (\ref{eq:proposedController}) and  the Riccati equation (\ref{eq:riccati}), we observe that the last expression is equal to $\valueMV$. This completes the proof.
\end{proof}

\section{Numerical Experiments}\label{sec:experiment}
This section presents the numerical experiments for an inverted pendulum on a cart. We verify that, under distributional changes, the proposed controller (\ref{eq:proposedController}) achieves a lower theoretical maximum value of the discounted cumulative cost (denoted in  (\ref{eq:MVC})), than the conventional discounted linear-quadratic regulator \cite[Section 4.3]{BertsekasDPVol1_4th}. As noted in Remark \ref{rem:solvability}, this conventional controller coincides with the proposed one as $\gamma\rightarrow\infty$.  Moreover, Theorem \ref{thm:MV} (i) and Remark \ref{rem:anyController} show that the value of  (\ref{eq:MVC}) under the proposed controller and the conventional controller, computed efficiently from the value functions in (\ref{eq:valueRiccati}) and (\ref{eq:valueLyapunov}), respectively, are  equal to the worst-case discounted cumulative cost. This equivalence follows from Corollary \ref{cor:equivalence} (ii) and Proposition \ref{prop:DRC}, provided that $\gamma$ is  sufficiently large such that the assumption (\ref{eq:assumption}) holds.

We consider an inverted pendulum on a cart ($\dimStat = 4$,$\,\dimDes=1$). The system matrices and parameters are given by: 
\begin{equation*}
\begin{split}
&A=
\begin{bmatrix}
1 && 0.1 && -0.0506 && -0.0017 \\
0 && 1 && -1.0240 && -0.0506 \\
0 && 0 && 1.0723 && 0.1024 \\
0 && 0 && 1.4628 && 1.0723
\end{bmatrix}
,\\
&B = 
\begin{bmatrix}
0.0106 && 0.202 && -0.007 && -0.146
\end{bmatrix}
^\top
, \\
&Q=10\times I
,\quad
R=1,\quad
\alpha = 0.985,\quad \\
&\dParameter \in \{10^5,3\times10^5,10^6,3\times10^6,10^7\}
,\\ 
&\Sigma=
\begin{bmatrix}
2 && 0.5 && 0 && 0 \\
0.5 && 3 && 0 && 0 \\
0 && 0 && 2 && 0.5 \\
0 && 0 && 0.5 && 3
\end{bmatrix}
.
\end{split}
\end{equation*}
We verified that (\ref{eq:assumption}) holds at each time step with probability at least $98\%$ for $\dParameter=10^5$, and probability one for larger $\dParameter$
, along trajectories generated by the system dynamics using $1000$ samples from a Gaussian distribution with covariance $\Sigma$, and the corresponding value functions in (\ref{eq:valueRiccati}) and (\ref{eq:valueLyapunov}).    

The results are presented in Figure \ref{fig:oneAndPendulum}. We observe that the proposed method achieves a lower theoretical maximum of the discounted cumulative cost than the conventional method.

\vspace{1pt}
\begin{figure}[t]
   \centering
   \includegraphics[scale=0.9,clip]{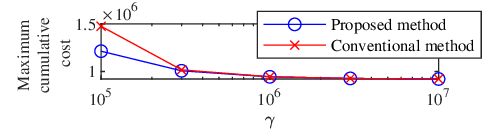}
   \caption{Numerical results for the inverted pendulum on a cart. The circles and crosses  represent the results of the proposed and conventional methods, respectively. The vertical axis represents the theoretical maximum of the discounted cumulative cost as shown in Corollary \ref{cor:equivalence} (ii), and the horizontal axis represents the coefficient for the distributional distance penalty in the distributionally robust control problem.}
   \label{fig:oneAndPendulum}
\end{figure}

\section{CONCLUSIONS}\label{sec:conc}
In this study, we propose a DRC method that avoids SIP and does not require the true system distribution, such as its moments. It extends existing theories on distributionally robust controller synthesis via the Bellman and Riccati equations, thereby simplifying both theory and computation.

\addtolength{\textheight}{-0cm}   




\bibliography{AvoidSIP}
\bibliographystyle{IEEEtran}

\end{document}